\documentclass[12pt]{amsart}
\usepackage[margin=1.4in]{geometry}
\usepackage{graphicx,bm}
\usepackage{slashed}
\usepackage[breaklinks,colorlinks]{hyperref}
\usepackage{verbatim}
\usepackage{cases}
\usepackage{empheq}
\usepackage{color}
\numberwithin{equation}{section}

\newtheorem{thm}{Theorem}[section]

\newtheorem{cor}[thm]{Corollary}
\newtheorem{lem}[thm]{Lemma}

\newtheorem{defn}[thm]{Definition}
\newtheorem{rmk}[thm]{Remark}

\newcommand{\pt}{\partial}
\DeclareMathOperator{\dive}{div}

\allowdisplaybreaks

\begin{document}

\title[Boundary  regularity of Dirac-harmonic maps]{Boundary partial regularity for a class of  Dirac-harmonic maps}
\author[J\"urgen Jost]{J\"urgen Jost}
\address{Max Planck Institute for Mathematics in the Sciences, Inselstrasse 22, 04103 Leipzig, Germany}
\email{jost@mis.mpg.de}

\author[Jingyong Zhu]{Jingyong Zhu}
\address{School of Mathematics, Sichuan University, Chengdu 610065, China}
\email{jzhu@scu.edu.cn}


% ----------------------------------------------------------------

\begin{abstract}
In this paper, we prove the boundary partial regularity for a class of coupled Dirac-harmonic maps satisfying a certain energy monotonicity inequality near the boundary.

\end{abstract}
\maketitle

% ---------------------------------------------------------------
\vspace{2em}
\section{Introduction}
Variational principles play an important role in both geometry and physics and one of the key problems with applications in both fields is the variational problem associated to the Dirichlet energy of maps between Riemannian manifolds. The critical points of this functional are called harmonic maps.

The Dirichlet energy is one of the simplest possible functionals involving first derivatives. In the theory of harmonic maps, however, it is studied for a nonlinear target, which gives rise to nonlinear Euler-Lagrange equations for its critical points.  While  it is not dificult to construct weak solutions, to show their regularity is much harder and in fact,  in general not possible. The harmonic map system belongs to a class for which no general regularity results hold. The best one can hope for is partial regularity.  One needs quite subtle arguments (see e.g. \cite{moser2005partial}\cite{lin2008analysis})
that use  the special structure of the equations or some rather deep results from harmonic analysis to prove (partial) regularity regularity. In turn, such methods then often also apply to  solutions of related problems, like perturbations and generalizations of the harmonic map problem or associated parabolic problems.

When the curvature is non-positive, one can use convexity arguments to obtain regularity (see \cite{jost2017riemannian} and the references therein).  
In contrast, weakly harmonic maps into manifolds whose sectional curvatures are either positive or change signs can behave quite wildly  in higher dimensions. Nevertheless, there are several important partial regularity theorems for  stationary harmonic maps (see e.g. \cite{bethuel1993singular}). These stationary harmonic maps satisfy the monotonicity formula, which plays an important role for proving  partial regularity. The boundary partial regularity is a more complicated  than the interior one. In \cite{wang1999boundary}, Wang proved the boundary partial regularity for a class of harmonic maps satisfying a certain energy monotonicity inequality near the boundary.

There are various generalizations of harmonic maps. Motivated by the supersymmetric nonlinear sigma model in quantum field theory, Dirac-harmonic maps from Riemann surfaces were introduced in \cite{chen2006dirac}. They are critical points of an unbounded conformally invariant functional involving two fields, a map from a Riemann surface into a Riemannian manifold and a section of a Dirac bundle which is the usual spinor bundle twisted with the pull-back of the tangent bundle of the target by the map. Since Dirac-harmonic maps are solutions to a coupled nonlinear elliptic system,  the regularity problem is challenging.

The interior regularity and partial regularity were proved in \cite{wang2009regularity} through the $\epsilon$-regularity theorem.
 In this paper, we prove the boundary partial regularity for coupled Dirac-harmonic maps (see the definition in the Section 2) satisfying a certain energy monotonicity inequality near the boundary.

\begin{thm}\label{holder}
Let $(u,\psi)$ be a coupled weakly  Dirac-harmonic map from a compact manifold  $M$ of dimension $m\geq2$ into another compact Riemannian manifold $N$ of dimension $n\geq2$, and $u=\phi\in C^{0,1}(\partial M,N)$ in the sense of trace. Suppose  $\psi\in L^{q}(M)$ for some $q>2m$ and  $u$ satisfies the energy monotonicity inequality
\begin{equation}\label{monotonicity}
    s^{2-m}\int_{M\cap B_s(x)}|\nabla u|^2\leq(1+C_0r+C_0s)r^{2-m}\int_{M\cap B_r(x)}|\nabla u|^2+C_0(r-s)
\end{equation}
for all $x\in M_{\rho_0}:=\{x\in\bar M: {\rm dist}(x,\partial M)< \rho_0\}$ and $0<s\leq r<\rho_0$. Here $C_0$ and $\rho_0$ depend only on $M$, $N$ and $\phi$.  
Then there exist $\alpha_0\in (0,1)$ and a closed set $\mathcal{S}\subset  M_{\rho_0}$ with $H^{m-2}(\mathcal{S})=0$ such that $u\in C^{\alpha_0}( M_{\rho_0}\setminus\mathcal{S},N)$. Moreover, if $\phi\in C^{1,1}(\partial M,N)$ and $q>\frac{4m}{\alpha_0}$, then $u\in C^{1,\mu}( M_{\rho_0}\setminus\mathcal{S},N)$ for some $\mu\in(0,1)$.
\end{thm}

\begin{rmk}
  The map part can be coupled with the spinor part which leads to difficulties not present for ordinary harmonic maps. In the coupled case, the tension field of the map component of a coupled Dirac-harmonic map is not zero and couples the curvature of the target manifold and the spinor field.  In the interoir of the domain, this term can be dealt with by applying the Coulomb gauge construction to the  Dirac-harmonic map equation. However, this method is no longer available to boundary points. Therefore, we need different techniques. \\
  Our scheme will develop very precise integral growth conditions for the fields of involved. We shall compare the solutions with those of various auxiliary equations, and to get these equations, we shall use H\'elein's construction of adapted tangent frames. 
\end{rmk}

Combining this theorem with the interior regularity result in \cite{wang2009regularity}, we have the following global regularity result.
\begin{cor}\label{}
Let $(u,\psi)$ be a coupled weakly Dirac-harmonic map from a compact manifold  $M$ of dimension $m\geq2$ into another compact Riemannian manifold $N$ of dimension $n\geq2$, and $u=\phi\in C^{0,1}(\partial M,N)$ in the sense of trace. Suppose  $\psi\in L^{q}(M)$ for some $q>2m$ and  $u$ satisfies \eqref{monotonicity} for all $x\in\bar{M}$ with $C_0$ and $\rho_0$ depend only on $M$, $N$ and $\phi$.
Then there exist $\alpha_0\in (0,1)$ and a closed set $\mathcal{S}\subset  \bar{M}$ with $H^{m-2}(\mathcal{S})=0$ such that $u\in C^{\alpha_0}(\bar{M}\setminus\mathcal{S},N)$. Moreover, if $\phi\in C^{1,1}(\partial M,N)$ and $q>\frac{4m}{\alpha_0}$, then $u\in C^{1,\mu}(\bar{M}\setminus\mathcal{S},N)$ for some $\mu\in(0,1)$.
\end{cor}

\begin{rmk}
  We should stress here that the coupled case is the difficult one. In the uncoupled case, the map part is simply harmonic, and so, the (partial) regularity for harmonic maps discussed above can be applied. \\
In any case, uncoupled or coupled, higher order regularity then follows from the standard elliptic theory. If $u\in W^{1,m}$, then $\mathcal{S}=\emptyset$.
%When the target manifold $N$ has no harmonic spheres (e.g. nonpositively curved), the singular set $\mathcal{S}$ in the results above satisfies $H^{m-3}(\mathcal{S})=0$. 
\end{rmk}

The rest of the paper is organized as follows: In Section 2, we recall the basic definitions and some facts about Dirac-harmonic maps. In Section 3, we show the small energy boundary regularity. In Section 4, we prove the Theorem \ref{holder} %and the Theorem \ref{smooth}. 

{\bf Acknowledgements}
The second author would like to thank the support by the National Natural Science Foundation of China (Grant No. 12201440) and the Fundamental Research Funds for the Central Universities.
% Our  gratitude goes to the anonymous reviewers for their careful work and thoughtful suggestions that have helped to improve this paper.

\vspace{2em}

\section{Preliminaries}

Let $(M, g)$ be a compact manifold of dimension $m\geq2$. On the complex spinor bundle $\Sigma M$, we denote the Hermitian inner product by $\langle\cdot, \cdot\rangle_{\Sigma M}$. For  $X\in\Gamma(TM)$ and $\xi\in\Gamma(\Sigma M)$, the Clifford multiplication is  skew-adjoint:
\begin{equation*}
\langle X\cdot\xi, \eta\rangle_{\Sigma M}=-\langle\xi, X\cdot\eta\rangle_{\Sigma M}.
\end{equation*}
Let $\nabla$ be the Levi-Civita connection on $(M,g)$. There is a unique  connection (also denoted by $\nabla$) on $\Sigma M$ compatible with $\langle\cdot, \cdot\rangle_{\Sigma M}$.  Choosing a local orthonormal basis $\{e_{\beta}\}_{\beta=1,2}$ on $M$, the usual Dirac operator is defined as $\slashed\partial:=e_\beta\cdot\nabla_\beta$, where $\beta=1,2$. Here and in the sequel, we use the Einstein summation convention. A reference fo  spin geometry is \cite{lawson1989spin}.

Let $u$ be a smooth map from $M$ to another compact Riemannian manifold $(N, h)$ of dimension $n\geq2$. Let $u^*TN$ be the pull-back bundle of $TN$ by $u$ and consider the twisted bundle $\Sigma M\otimes u^*TN$. On this bundle there is a metric $\langle\cdot,\cdot\rangle_{\Sigma M\otimes u^*TN}$ induced from the metric on $\Sigma M$ and $u^*TN$. Also, we have a connection $\tilde\nabla$ on this twisted bundle naturally induced from those on $\Sigma M$ and $u^*TN$. In local coordinates $\{y^\mu\}_{\mu=1,\dots,n}$, the section $\psi$ of $\Sigma M\otimes u^*TN$ is written as
$$\psi=\psi^\mu\otimes\partial_{y^\mu}(\phi),$$
where each $\psi^\mu$ is a usual spinor on $M$. We also have the following local expression of $\tilde\nabla$
$$\tilde\nabla\psi=\left(\nabla\psi^\mu+\Gamma_{\lambda\sigma}^\mu(u)\nabla \phi^\lambda\cdot\psi^\sigma\right)\otimes\partial_{y^\mu}(u),$$
where $\Gamma^\mu_{\lambda\sigma}$ are the Christoffel symbols of the Levi-Civita connection of $N$. The Dirac operator along the map $\phi$ is defined as
\begin{equation}\label{dirac}
\slashed{D}\psi:=e_\alpha\cdot\tilde\nabla_{e_\alpha}\psi=\left(\slashed\partial\psi^\mu+\Gamma_{\lambda\sigma}^\mu(u)\nabla_{e_\alpha}u^\lambda(e_\alpha\cdot\psi^\sigma)\right)\otimes\partial_{y^\mu}(u),
\end{equation}
which is self-adjoint (see \cite{jost2017riemannian}). Sometimes, we use $\slashed{D}^u$ to distinguish the Dirac operators along different maps. In \cite{chen2006dirac}, the authors  introduced the  functional
\begin{equation*}\begin{split}
L(\phi,\psi)&:=\frac12\int_M\left(|du|^2+\langle\psi,\slashed{D}\psi\rangle_{\Sigma M\otimes u^*TN}\right)\\
&=\frac12\int_M\left( h_{\lambda\sigma}(u)g^{\alpha\beta}\frac{\partial u^\lambda}{\partial x^\alpha}\frac{\pt u^\sigma}{\pt x^\beta}+h_{\lambda\sigma}(u)\langle\psi^\lambda,\slashed{D}\psi^\sigma\rangle_{\Sigma M}\right).
\end{split}
\end{equation*}
They computed the Euler-Lagrange equations of $L$:
\begin{numcases}{}
   \tau^\kappa(u)-\frac12R^\kappa_{\lambda\mu\sigma}\langle\psi^\mu,\nabla u^\lambda\cdot\psi^\sigma\rangle_{\Sigma M}=0,  \label{eldh1}\\
   \slashed{D}\psi^\mu:=\slashed\partial\psi^\mu+\Gamma_{\lambda\sigma}^\mu(u)\nabla_{e_\alpha}u^\lambda(e_\alpha\cdot\psi^\sigma)=0, \label{eldh2}
\end{numcases}
where $\tau^\kappa(u)$ is the $\kappa$-th component of the tension field \cite{jost2017riemannian} of the map $u$ with respect to the coordinates on $N$, $\nabla u^\lambda\cdot\psi^\sigma$ denotes the Clifford multiplication of the vector field $\nabla u^\lambda$ with the spinor $\psi^\sigma$, and $R^\kappa_{\lambda\mu\sigma}$ stands for the components of the Riemann curvature tensor of the target manifold $N$. By denoting
\begin{equation}\label{curvature term}
    \mathcal{R}(u,\psi):=\frac12R^\kappa_{\lambda\mu\sigma}\langle\psi^\mu,\nabla u^\lambda\cdot\psi^\sigma\rangle_{\Sigma M}\pt_{y^\kappa},
\end{equation}
 we can write \eqref{eldh1} and \eqref{eldh2} in the following global form:
\begin{numcases}{}
\tau(u)=\mathcal{R}(u,\psi), \label{geldh1} \\
\slashed{D}\psi=0,  \label{geldh2}
\end{numcases}
and call the solutions $(\phi,\psi)$ Dirac-harmonic maps from $M$ to $N$. Moreover, if $\mathcal{R}(u,\psi)\neq0$, the Dirac-harmonic map $(\phi,\psi)$ is called coupled.

By \cite{nash1956imbedding}, we can isometrically embed $N$ into a Euclidean space $\mathbb{R}^K$ for some constant $K$. Then \eqref{geldh1}-\eqref{geldh2} is equivalent to the  following system:
\begin{numcases}{}
		\Delta_g{u}=II(du,du)+Re(P(\mathcal{S}(du(e_\beta),e_{\beta}\cdot\psi);\psi)), \label{map eq}\\
		\slashed{\partial}\psi=\mathcal{S}(du(e_\beta),e_{\beta}\cdot\psi).
\end{numcases}
Here $II$ is the second fundamental form of $N$ in $\mathbb{R}^K$,
\begin{equation*}
\mathcal{S}(du(e_\beta),e_{\beta}\cdot\psi):=(\nabla{u^i}\cdot\psi^j)\otimes II(\partial_{z^i},\partial_{z^j}),
\end{equation*}
\begin{equation*}
Re(P(\mathcal{S}(du(e_\beta),e_{\beta}\cdot\psi);\psi)):=P(S(\partial_{z^k},\partial_{z^j});\partial_{z^i})Re(\langle\psi^i,du^k\cdot\psi^j\rangle).
\end{equation*}
where $i,j,k=1,\dots,K$, $P(\xi;\cdot)$ denotes the shape operator, defined by
\begin{equation*}
  \nabla_X\xi=-P(\xi;X)+\nabla^\perp\xi, \forall X\in \Gamma(TN)
\end{equation*}
and $Re(\zeta)$ denotes the real part of $\zeta\in\mathbb{C}$.

Here, we are interested in the regularity at the boundary. Therefore, we assume that 
 $M$ has a boundary, and we need to impose   well-defined and natural  boundary conditions for both the map and the spinor. For the map component, one can still use the Dirichlet boundary condition. For the spinor part, the chiral boundary operator ${\bf B}={\bf B}^\pm$ was first introduced in \cite{chen2013boundary} as follows:
\begin{equation}\label{bdy op}
\begin{split}
{\bf B}^\pm: L^2(\partial M, (\Sigma M\otimes u^*TN)|_{\partial M})&\to L^2(\partial M, (\Sigma M\otimes u^*TN)|_{\partial M})\\
\psi&\mapsto\frac12(Id\otimes Id\pm{\bf n}\cdot G\otimes Id)\psi
\end{split}\end{equation}
where ${\bf n}$ is the outward unit normal vector field on $\partial M$ and $G$ is the chirality operator for  spinors  introduced by Gibbons-Hawking-Horowitz-Perry \cite{Gibbons1983Positive}. It satisfies
\begin{equation}
G^2=Id, \ G^*=G, \ \nabla G=0, G\cdot X\cdot=-X\cdot G, \forall X\in\Gamma(TM).
\end{equation}
%-------------------------------------------------------------------------------------------------------------------------------------------------------------
%\begin{comment}

Our analytical strategy will involve precise integral growth estimates. 
 Important function spaces with such integral growth conditions are the  Morrey spaces.

\begin{defn}
For $p\geq1$, $0<\lambda\leq m$ and a domain $\Omega\subset \mathbb{R}^m$, the Morrey space $M^{p,\lambda}(\Omega)$ is defined by
\begin{equation}
M^{p,\lambda}(\Omega):=\{f\in L^p_{loc}(\Omega): \|f\|_{M^{p,\lambda}(\Omega)}<+\infty\},
\end{equation}
where
\begin{equation}
\|f\|_{M^{p,\lambda}}(\Omega)=\sup\{r^{\lambda-m}\int_{B_r(y)}|f|^p: B_r(y)\subset\Omega\}.
\end{equation}
\end{defn}
For certain ranges of $m,p,\lambda$, they embed into H\"older spaces, but for our problem, we are at a borderline situation and therefore need to derive  more refined estimates. 
%\end{comment}
%-------------------------------------------------------------------------------------------------------------------------------------------------------------

\vspace{2em}

\section{Small energy boundary regularity}

The goal of this section is to prove the following small energy boundary regularity result.
\begin{thm}\label{sebr}
Suppose the  assumptions of Theorem \ref{holder} hold.  
There exist $\varepsilon_0>0$, $\delta_0\in (0,\rho_0)$ and $\alpha_0\in (0,1)$ such that if $u$ satisfies
\begin{equation}
    r_0^{2-m}\int_{M\cap B_{r_0}(x_0)}|\nabla u|^2\leq \varepsilon_0^2
\end{equation}
for some $x_0\in \partial M$ and $r_0\in(0,\delta_0)$, then
 $u\in C^{\alpha_0}( \bar{M}\cap B_{r_0/2}(x_0),N)$.
\end{thm}

Since Theorem \ref{sebr} is a local result, it suffices to prove the result on a domain in Euclidean space $\mathbb{R}^m$. Therefore, for simplicity, we assume that $M=\Omega\subset \mathbb{R}^m$, and define
\begin{equation}
M(x,r)=\sup\{\rho^{1-m}\int_{\Omega\cap B_\rho(y)}|Du|: B_\rho(y)\subset B_r(x)\},
\end{equation}
%\begin{equation}
%\|f\|_{M^{p,\lambda}}(x,r)=\sup\{\rho^{\lambda-m}\int_{\Omega\cap B_\rho(y)}|f|^p: B_\rho(y)\subset B_r(x)\}.
%\end{equation}

Note that if $u$ satisfies \eqref{monotonicity}, then $M(x,r)$ is bounded for $x\in \Omega_{\rho_0}$ and $r\in (0,{\rho_0}/3)$. The key step in proving Theorem \ref{sebr} is the following lemma.

\begin{lem}\label{M}
Again,  the same assumptions as in Theorem \ref{holder}.  
There exist $\varepsilon_1>0$, $\theta_1\in(0,1/4)$, and $\delta_1\in (0,\rho_0)$ such that for $x_1\in \Omega_{\delta_1}$ and $r_1\in(0,\delta_1)$, if 
\begin{equation}
    (2r_1)^{2-m}\int_{M\cap B_{2r_1}(x_1)}|\nabla u|^2\leq \varepsilon_1^2,
\end{equation}
then 
\begin{equation}\label{M monotone}
M(x_1,\theta_1r_1)\leq \frac14 M(x_1,r_1)+C_1r_1+r_1^{1-\frac{2m}{q}}.
\end{equation}
\end{lem}

To prove the Lemma \ref{M}, let us recall the adapted frame constructions due to H{\'e}lein. Following \cite{heleinregularite}, we always assume that there exists a global orthonormal tangent frame $\{e_l\}_{l=1}^n$ adapted to $u\in H^1(\Omega, N)$ (i.e. $e_l(x)\in T_{u(x)}N$ for a.e. $x\in \Omega$, $1\leq l\leq n$), which satisfies
\begin{equation}
\sum_l\int_\Omega|De_l|^2\leq C_3\int_{\Omega}|Du|^2,
\end{equation}
\begin{equation}
{\dive} \langle  De_i,e_j\rangle=0, \ \text{for } \  1\leq i,j\leq n,
\end{equation}
\begin{equation}
{\dive} \langle  e_i,\frac{\pt e_j}{\pt\nu}\rangle=0, \ \text{for } \  1\leq i,j\leq n,
\end{equation}
where $\nu$ is the outward unit normal of $\pt \Omega$. For simplicity, we assume that $x_0=0$ and $\Omega=B_1^+$. For $r\in(0,1] $, denote $B_r^+(0)=\{x=(x',x_m)\in B_r(0): x_m\geq0 \}$, $T_r=\{x=(x',x_m)\in B_r^+(0): x_m=0\}$.
 Let $\{e\}_{l=1}^n: B_1^+\to \mathbb{R}^K$ be obtained as above, we extend it to $B$ as follows.
 \begin{equation}
 \bar{e}_l(x',x_m)=\left\{
 \begin{aligned}
		&e_l(x',x_m), \  x_m\geq0,\\
		&e_l(x',-x_m), \ x_m\leq 0.
\end{aligned}
\right.
 \end{equation}
 Then we have
 \begin{equation}
 \sum_{l}\int_{B_1}|D\bar{e}_l|^2\leq C_4\int_{B_1^+}|Du|^2,
 \end{equation}
  \begin{equation}
{\dive} \langle  D\bar{e}_i,\bar{e}_j\rangle=0, \ \text{for } \  1\leq i,j\leq n,
 \end{equation}
 Now, we extend $\phi$ to $B_1^+$ by letting $\phi(x',x_m)=\phi(x')$, and extend $u-\phi$ from $B_1^+$ to $B_1$ oddly with respect to $x_m$,
 \begin{equation}
 u_{\phi}(x',x_m)=\left\{
 \begin{aligned}
&(u-\phi)(x',x_m), \ x_m\geq0,\\
&-(u-\phi)(x',-x_m), \ x_m\leq 0.
\end{aligned}
\right.
 \end{equation}
 
 Now we are ready to prove the following lemma.
 
 \begin{lem}\label{boundary M}
 Under the same assumptions as in Theorem \ref{holder},  
 there exist $\varepsilon_2>0$, $\theta_2\in(0,1/2)$, and $\delta_2\in (0,\rho_0)$ such that for $r_2\in(0,\delta_2)$, if 
\begin{equation}
    r_2^{2-m}\int_{B_{r_2}^+(0)}|\nabla u|^2\leq \varepsilon_2^2,
\end{equation}
then for $x\in T_{r_2/2}$, $\rho_1\in (0,r_2/2)$,
\begin{equation}\label{bdy}
(\theta_2 \rho_1)^{1-m}\int_{B^+_{\theta_2 \rho_1}(x)}|Du|\leq \frac{1}{6^m} M(x,\rho_1)+Cr_2+ \frac{1}{6^{m+1}}r_2^{1-\frac{2m}{q}}.
\end{equation}
 \end{lem}

\begin{proof}
Let us first prove \eqref{bdy} for $x=0$. For $\rho\in(0,r_2)$, let $\eta\in C_0^\infty(B_\rho)$ be such that $\eta=1$ in $B_{\rho/2}$, $\eta=0$ outside $B_\rho$, and $|D\eta|\leq2/\rho$. It follows from the HDR theorem in \cite{bethuel1993singular} that there exist $\alpha_l\in H^1(\mathbb{R}^m,\mathbb{R}^K)$ and $\beta_l\in H^1(\mathbb{R}^m,\wedge^2\mathbb{R}^K)$ such that
\begin{equation*}
d^*\alpha_l=d\beta_l=0,
\end{equation*}
\begin{equation}\label{eq u phi}
\bar{e}_l\wedge d(\eta u_\phi)=d\alpha_l+d^*\beta_l
\end{equation}
and 
\begin{equation}
\|\alpha_l\|_{H^1}+\|\beta_l\|_{H^1}\leq C_6\|D(\eta u_\phi)\|_{L^2}\leq C_7(\|Du\|_{L^2(B_\rho^+)}+\|D\phi\|_{L^2(B_\rho^+)}),
\end{equation}
where we have used the Poincar{\'e} inequality. It follows from \eqref{eq u phi} that we can  control the left-hand side of \eqref{bdy} by  suitable estimates on $\alpha_l$ and $\beta_l$. To do so, we differentiate \eqref{eq u phi} and get
\begin{equation}\label{beta}
\Delta \beta_l^{ij}=\{\bar{e}_l,\eta u_\phi\}_{ij}=(\bar{e}_l)_i(\eta u_\phi)_j-(\bar{e}_l)_j(\eta u_\phi)_i,
\end{equation}
where $\beta_l^{ij}$ are the components of $\beta_l$ and 
\begin{equation}\label{alpha}
\Delta \alpha_l=\dive\langle Du_\phi, \bar{e}_l\rangle, \ \text{on} \ B_{\rho/2}.
\end{equation}

Multiplying \eqref{beta} by $\beta_l^{ij}$ and integrating on $\mathbb{R}^m$, we get
\begin{equation}
\|\nabla \beta_l\|_{L^2(\mathbb{R}^m)}\leq C\|\eta u_\phi\|_{BMO}\|\nabla u\|_{L^2(B^+_\rho)}.
\end{equation}
Moreover, 
\begin{equation}
\|\eta u_{\phi}\|_{BMO}\leq CM(0,\rho)\leq CM(0,r_2).
\end{equation}
By \eqref{monotonicity}, we have
\begin{equation}
\rho^{2-m}\int_{B_\rho}|\nabla \beta|^2\leq CM^2(0,\rho)(\varepsilon_2^2+r_2).
\end{equation}
Therefore, 
\begin{equation}\label{Dbeta}
\rho^{1-m}\int_{B_\rho}|\nabla \beta|\leq CM(0,\rho)(\varepsilon_2+\sqrt{r_2}).
\end{equation}

Based on \eqref{alpha}, we have the following equation for $\alpha_l$:
\begin{equation}\label{alpha2}
\dive(D\alpha_l-\langle \omega_\phi,\bar{e}_l\rangle)=\langle Du_\phi, D\bar{e}_l\rangle-\langle \omega_\phi, D\bar{e}_l\rangle
+\frac12R^l_{kij}(\bar{u})\langle \bar{\psi}^i, U^k\cdot\bar{\psi}^j\rangle,
\end{equation}
where
\begin{equation}\label{omega phi}
\omega_\phi=\left\{
 \begin{aligned}
-D\phi(x',x_m), \ x_m\geq0,\\
D\phi(x',-x_m), \ x_m< 0,
\end{aligned}
\right.
\end{equation}
\begin{equation}
\bar{u}=\left\{
 \begin{aligned}
&u(x',x_m), \ x_m\geq0,\\
&u(x',-x_m), \ x_m< 0.
\end{aligned}
\right.
\end{equation}
\begin{equation}
U^k=\left\{
 \begin{aligned}
&\nabla u^k(x',x_m), \ x_m\geq0,\\
&-\nabla u^k(x',-x_m), \ x_m< 0,
\end{aligned}
\right.
\end{equation}
and 
\begin{equation}\label{bar psi}
\bar{\psi}^i=\left\{
 \begin{aligned}
&\psi^i(x',x_m), \ x_m\geq0,\\
&\psi^i(x',-x_m), \ x_m< 0.
\end{aligned}
\right.
\end{equation}

Since all the functions can be decomposed into even and odd parts, the proof of \eqref{alpha2} can be devided into two cases as follows.

\textit{Case 1.} \ Suppose $\varphi\in C^\infty_0(B_{\rho/2})$ is even, i.e. $\varphi(x',-x_m)=\varphi(x',x_m)$. By \eqref{alpha}, we have
\begin{equation}
\int_{B_{\rho/2}}\langle D\alpha_l, D\varphi\rangle%=-\int_{B_{\rho/2}}\langle\Delta\alpha_l,\varphi\rangle  
=\int_{B_{\rho/2}}\langle Du_\phi, \bar{e}_l\rangle D\varphi=0.
\end{equation}
On the other hand, it follows from \eqref{omega phi}-\eqref{bar psi} that
\begin{equation}
\begin{split}
&\int_{B_{\rho/2}}\langle Du_\phi, D\bar{e}_l\rangle \varphi=\int_{B_{\rho/2}}\langle \bar{e}_l, \omega_\phi\rangle D\varphi\\
&=\int_{B_{\rho/2}}\langle D\bar{e}_l, \omega_\phi\rangle \varphi=\int_{B_{\rho/2}}R^l_{kij}(\bar{u})\langle \bar{\psi}^i, U^k\cdot\bar{\psi}^j\rangle\varphi=0.
\end{split}
\end{equation}
This implies \eqref{alpha2}.

\textit{Case 2.} \ Suppose $\varphi\in C^\infty_0(B_{\rho/2})$ is odd, i.e. $\varphi(x',-x_m)=-\varphi(x',x_m)$ for $x_m<0$. By \eqref{alpha}, we have
\begin{equation}
\int_{B_{\rho/2}}\langle D\alpha_l, D\varphi\rangle=2\int_{B^+_{\rho/2}}\langle Du_\phi, \bar{e}_l\rangle D\varphi.
\end{equation}

For $\varepsilon>0$, we define
\begin{equation}\label{eta epsilon}
\eta_{\varepsilon}(x',x_m)=\left\{
 \begin{aligned}
1, \ |x_m|\geq2\varepsilon,\\
0, \ \  |x_m|<\varepsilon,
\end{aligned}
\right.
\end{equation}and $|D\eta_{\varepsilon}|\leq \frac2\varepsilon$.
Then  
\begin{equation}
 \begin{aligned}
\int_{B^+_{\rho/2}}\langle Du_\phi, \bar{e}_l\rangle D\varphi&=\lim_{\varepsilon\to 0}\int_{B^+_{\rho/2}}\langle Du_\phi, \bar{e}_l\rangle D\varphi\eta_\varepsilon\\
&=\lim_{\varepsilon\to 0}\int_{B^+_{\rho/2}}\langle Du_\phi, \bar{e}_l\rangle (D(\varphi\eta_\varepsilon)-D\eta_{\varepsilon}\varphi).
\end{aligned}
\end{equation}
 Since $\varphi(x',0)=0$, we have
\begin{equation}\label{deta}
 \begin{aligned}
&\lim_{\varepsilon\to 0}\left|\int_{B^+_{\rho/2}}\langle Du_\phi, \bar{e}_l\rangle D\eta_{\varepsilon}\varphi\right|\\
&\leq\lim_{\varepsilon\to 0}\frac{2}{\varepsilon}\int_{B^+_{\rho/2}\cap \{x: \varepsilon\leq x_m\leq2\varepsilon\}}|Du_\phi||D\varphi||x_m|\\
&\leq\lim_{\varepsilon\to 0} 2\int_{B^+_{\rho/2}\cap \{x: \varepsilon\leq x_m\leq2\varepsilon\}}|Du_\phi||D\varphi|\\
&= 0.
\end{aligned}
\end{equation}
Hence, we get
\begin{equation}
 \begin{aligned}
\int_{B^+_{\rho/2}}\langle Du_\phi, \bar{e}_l\rangle D\varphi
&=\lim_{\varepsilon\to 0}\int_{B^+_{\rho/2}}\langle Du_\phi, \bar{e}_l\rangle D(\varphi\eta_\varepsilon)\\
&=\lim_{\varepsilon\to 0}\int_{B^+_{\rho/2}}\langle Du-D\phi, \bar{e}_l\rangle D(\varphi\eta_\varepsilon).
\end{aligned}
\end{equation}
It follows from the equation of \eqref{geldh1} that
\begin{equation}
 \begin{aligned}
&\lim_{\varepsilon\to 0}\int_{B^+_{\rho/2}}\langle Du, \bar{e}_l\rangle D(\varphi\eta_\varepsilon)\\
&=-\lim_{\varepsilon\to 0}\int_{B^+_{\rho/2}}\dive\langle Du, \bar{e}_l\rangle \eta_\varepsilon\varphi\\
&=-\lim_{\varepsilon\to 0}\int_{B^+_{\rho/2}}\langle Du, D\bar{e}_l\rangle \eta_\varepsilon\varphi-\frac12\lim_{\varepsilon\to 0}\int_{B^+_{\rho/2}}R^l_{kij}\langle\psi^i,\nabla u^k\cdot\psi^j\rangle\eta_\varepsilon\varphi.
\end{aligned}
\end{equation}
On the other hand, similar to  \eqref{deta}, we get
\begin{equation}
\lim_{\varepsilon\to 0}\int_{B^+_{\rho/2}}\langle D\phi, \bar{e}_l\rangle D\eta_\varepsilon\varphi=0.
\end{equation}
Therefore, we obtain
\begin{equation}
 \begin{aligned}
\int_{B_{\rho/2}}\langle Du_\phi, \bar{e}_l\rangle D\varphi
=&-2\int_{B^+_{\rho/2}}\langle Du, D\bar{e}_l\rangle\varphi-2\int_{B^+_{\rho/2}}\langle D\phi, \bar{e}_l\rangle D\varphi\\
&-\int_{B^+_{\rho/2}}R^l_{kij}\langle\psi^i,\nabla u^k\cdot\psi^j\rangle\varphi\\
=&-2\int_{B^+_{\rho/2}}\langle Du_\phi, D\bar{e}_l\rangle\varphi-2\int_{B^+_{\rho/2}}\langle D\phi, D\bar{e}_l\rangle\varphi\\
&-2\int_{B^+_{\rho/2}}\langle D\phi, \bar{e}_l\rangle D\varphi
-\int_{B^+_{\rho/2}}R^l_{kij}\langle\psi^i,\nabla u^k\cdot\psi^j\rangle\varphi\\
=&-\int_{B_{\rho/2}}\langle Du_\phi, D\bar{e}_l\rangle\varphi+\int_{B_{\rho/2}}\langle \omega_\phi, D\bar{e}_l\rangle\varphi\\
&+\int_{B_{\rho/2}}\langle \omega_\phi, \bar{e}_l\rangle D\varphi
-\frac12\int_{B_{\rho/2}}R^l_{kij}(\bar{u})\langle \bar{\psi}^i, U^k\cdot\bar{\psi}^j\rangle\varphi\\
=&-\int_{B_{\rho/2}}\langle Du_\phi, D\bar{e}_l\rangle\varphi+\int_{B_{\rho/2}}\langle \omega_\phi, D\bar{e}_l\rangle\varphi\\
&-\int_{B_{\rho/2}}\dive\langle \omega_\phi, \bar{e}_l\rangle \varphi
-\frac12\int_{B_{\rho/2}}R^l_{kij}(\bar{u})\langle \bar{\psi}^i, U^k\cdot\bar{\psi}^j\rangle\varphi.
\end{aligned}
\end{equation}
This implies
\begin{equation}
 \begin{aligned}
\int_{B_{\rho/2}}\dive(D\alpha_l)\varphi
=&\int_{B_{\rho/2}}\langle Du_\phi, D\bar{e}_l\rangle\varphi-\int_{B_{\rho/2}}\langle \omega_\phi, D\bar{e}_l\rangle\varphi\\
&+\int_{B_{\rho/2}}\dive\langle \omega_\phi, \bar{e}_l\rangle \varphi
+\frac12\int_{B_{\rho/2}}R^l_{kij}(\bar{u})\langle \bar{\psi}^i, U^k\cdot\bar{\psi}^j\rangle\varphi,
\end{aligned}
\end{equation}
and  \eqref{alpha2} follows.

Now, we are ready to give estimates on $\alpha_l$. Let $\alpha_l=\alpha_l^1+\alpha_l^2$, where
\begin{equation}
\left\{
 \begin{aligned}
 &\Delta\alpha_l^1=0, \ \text{in} \ B_{\rho/2}\\
  &\alpha_l^1=\alpha_l, \ \text{on} \ \pt B_{\rho/2}
\end{aligned}
\right.
\end{equation}
and
\begin{equation}
\left\{
 \begin{aligned}
\dive(D\alpha_l^2-\langle \omega_\phi,\bar{e}_l\rangle)=&\langle Du_\phi, D\bar{e}_l\rangle-\langle \omega_\phi, D\bar{e}_l\rangle\\
&+\frac12R^l_{kij}(\bar{u})\langle \bar{\psi}^i, U^k\cdot\bar{\psi}^j\rangle, \ \text{in} \ B_{\rho/2}\\
\alpha_l^2=0, \ \ \text{on} \ \pt B_{\rho/2}.
\end{aligned}
\right.
\end{equation}
Let $\tilde\rho=\rho/2$. Since $\alpha_l^1$ is harmonic, by the mean-value inequality, we have for $\forall \theta_0\in(0,1)$
\begin{equation}\label{dalpha1}
 \begin{aligned}
&(\theta_0\tilde\rho)^{1-m}\int_{B_{\theta_0\tilde\rho}}|D\alpha_l^1|\\
&\leq\theta_0\tilde\rho^{1-m}\int_{B_{\tilde\rho}}|D\alpha_l^1|\\
&\leq\theta_0\tilde\rho^{1-m}\int_{B_{\tilde\rho}}|Du|+|D\phi|+|D\beta_l|+|D\alpha_l^2|\\
&\leq C\theta_0 M(0,\tilde\rho)+C\theta_0\tilde\rho+C\theta_0 M(0,\rho)(\varepsilon_2+\sqrt{r_2})+\theta_0\tilde\rho^{1-m}\int_{B_{\tilde\rho}}|D\alpha_l^2|\\
&\leq C\theta_0 r_2+C\theta_0 M(0,\rho)+\theta_0\tilde\rho^{1-m}\int_{B_{\tilde\rho}}|D\alpha_l^2|.
\end{aligned}
\end{equation}
Next, we derive the estimate for $\tilde\rho^{1-m}\int_{B_{\tilde\rho}}|D\alpha_l^2|$. To do so, we define 
\begin{equation}
\left\{
 \begin{aligned}
&\Delta v_l=\dive\left(\frac{D\alpha_l^2}{|D\alpha_l^2|}\right), \ \text{in} \ B_{\tilde\rho}\\
&v_l=0, \ \text{on} \ \pt B_{\tilde\rho}.
\end{aligned}
\right.
\end{equation}
So we have
\begin{equation}
\|Dv_l\|_{L^2(B_{\tilde\rho})}\leq C\tilde\rho^{\frac{m}{2}},  \  \|v_l\|_{L^\infty(B_{\tilde\rho})}\leq C\tilde\rho.
\end{equation}
Since $\alpha_l^2=v_l=0$ on $\pt B_{\tilde\rho}$, we have
\begin{equation}
 \begin{aligned}
&\int_{B_{\tilde\rho}}|D\alpha_l^2|\\
=&-\int_{B_{\tilde\rho}}v_l\Delta\alpha_l^2\\
=&-\int_{B_{\tilde\rho}}v_l\Bigg(\langle Du_\phi, D\bar{e}_l\rangle-\langle \omega_\phi, D\bar{e}_l\rangle+\dive\langle \omega_\phi, \bar{e}_l\rangle \\
&\quad\quad\quad\quad\quad\quad+\frac12R^l_{kij}(\bar{u})\langle \bar{\psi}^i, U^k\cdot\bar{\psi}^j\rangle\Bigg)\\
=&-\int_{B_{\tilde\rho}}\langle Du_\phi, D\bar{e}_l\rangle v_l+\int_{B_{\tilde\rho}}\langle  \omega_\phi, D\bar{e}_l\rangle v_l+\int_{B_{\tilde\rho}}\langle  \omega_\phi , \bar{e}_l\rangle Dv_l\\
&-\frac12\int_{B_{\tilde\rho}}R^l_{kij}(\bar{u})\langle \bar{\psi}^i, U^k\cdot\bar{\psi}^j\rangle v_l.
\end{aligned}
\end{equation}
Note that
\begin{equation}
 \begin{aligned}
&-\int_{B_{\tilde\rho}}\langle Du_\phi, D\bar{e}_l\rangle v_l\\
&=-\int_{B_{\tilde\rho}}\langle \omega_\phi, D\bar{e}_l\rangle v_l-\left(\int_{B^+_{\tilde\rho}}\langle Du, D\bar{e}_l\rangle v_l-\int_{B^-_{\tilde\rho}}\langle D\bar{u}, D\bar{e}_l\rangle v_l\right),
\end{aligned}
\end{equation}
we have
\begin{equation}
 \begin{aligned}
&\int_{B_{\tilde\rho}}|D\alpha_l^2|\\
=&-\left(\int_{B^+_{\tilde\rho}}\langle Du, D\bar{e}_l\rangle v_l-\int_{B^-_{\tilde\rho}}\langle D\bar{u}, D\bar{e}_l\rangle v_l\right)+\int_{B_{\tilde\rho}}\langle  \omega_\phi , \bar{e}_l\rangle Dv_l\\
&-\frac12\int_{B_{\tilde\rho}}R^l_{kij}(\bar{u})\langle \bar{\psi}^i, U^k\cdot\bar{\psi}^j\rangle v_l\\
=:&I+II+III.
\end{aligned}
\end{equation}

We estimate $I,II,III$ as follows.
\begin{equation}\label{II}
|II|\leq C {\rm Lip}(\phi)\|Dv_l\|_{L^2(B_{\tilde\rho})}\tilde\rho^{\frac{m}{2}}\leq C\tilde\rho^m.
\end{equation}
\begin{equation}\label{III}
 \begin{aligned}
|III|&\leq C\|v_l\|_{L^\infty(B_{\tilde\rho})}\int_{B^+_{\tilde\rho}}|\psi|^2|\nabla u|\\
&\leq C\tilde\rho^{m-1}\left(\tilde\rho^{2-m}\int_{B^+_{\tilde\rho}}|\nabla u|^2\right)^{\frac12}\left(\tilde\rho^{2-m}\int_{B^+_{\tilde\rho}}|\psi|^4\right)^{\frac12}\\
&\leq C\tilde\rho^{m-\frac{2m}{q}}(\varepsilon_2+\sqrt{r_2})\|\psi\|^2_{L^q}.
\end{aligned}
\end{equation}
\begin{equation}
 \begin{aligned}
I&=-\left(\int_{B^+_{\tilde\rho}}\langle Du, D\bar{e}_l\rangle v_l-\int_{B^-_{\tilde\rho}}\langle D\bar{u}, D\bar{e}_l\rangle v_l\right)\\
&=-\int_{B^+_{\tilde\rho}}\langle Du, D\bar{e}_l\rangle \bar{v}_l,
\end{aligned}
\end{equation}
where $\bar{v}_l(x',x_m)=v_l(x',x_m)-v_l(x',-x_m)$. Since $\bar{v}_l=0$ on $\pt B^+_{\tilde\rho}$, we get
\begin{equation}
 \begin{aligned}
I&=-\int_{B^+_{\tilde\rho}}\langle Du, D\bar{e}_l\rangle \bar{v}_l\\
&=\int_{B^+_{\tilde\rho}}(-1)^m\sum_{i<j}\{D_{m,l}^{ij},u\}_{ij}\bar{e}_m\bar{v}_l\\
&=(-1)^{m+1}\int_{B^+_{\tilde\rho}}\sum_{i<j}\{D_{m,l}^{ij},\bar{e}_m\bar{v}_l\}_{ij}u,
\end{aligned}
\end{equation}
where $D_{m,l}$ is a $2$-form in $H^1(B_1\subset\mathbb{R}^m,\mathbb{R}^K)$ such that
\begin{equation}
d^*D_{m,l}=\langle D\bar{e}_l,\bar{e}_m\rangle
\end{equation}
and
\begin{equation}
\int_{B^+_{\tilde\rho}}|\nabla D_{m,l}|^2\leq C\int_{B^+_{\tilde\rho}}|\nabla u|^2\leq C(\varepsilon^2_2+r_2)\tilde\rho^{m-2}.
\end{equation}
Hence,
\begin{equation}\label{I}
 \begin{aligned}
|I|&\leq C\|\nabla D_{m,l}\|_{L^2(B^+_{\tilde\rho})}\|\nabla \bar{e}_m\bar{v}_l\|_{L^2(B^+_{\tilde\rho})}\|u\|_{BMO}\\
&\leq C\tilde\rho^{m-1}(\varepsilon_2+\sqrt{r_2})M(0,\rho).
\end{aligned}
\end{equation}
Combining \eqref{I}, \eqref{II} and \eqref{III}, we get
\begin{equation}\label{Dalpha2}
\tilde\rho^{1-m}\int_{B_{\tilde\rho}}|D\alpha_l^2|\leq C\Bigg(r_2+(\varepsilon_2+\sqrt{r_2})M(0,\rho)+r_2^{1-\frac{2m}{q}}(\varepsilon_2+\sqrt{r_2})\Bigg).
\end{equation}
Plugging this into \eqref{dalpha1}, we have
\begin{equation}\label{Dalpha1}
 \begin{aligned}
&(\theta_0\tilde\rho)^{1-m}\int_{B_{\theta_0\tilde\rho}}|D\alpha_l^1|\\
&\leq C\theta_0 r_2+C\theta_0M(0,r_2)+\theta_0\tilde\rho^{1-m}\int_{B_{\tilde\rho}}|D\alpha_l^2|\\
&\leq C\theta_0r_2+C\theta_0(1+\varepsilon_2+\sqrt{r_2})M(0,\rho)+C\theta_0r_2^{1-\frac{2m}{q}}(\varepsilon_2+\sqrt{r_2})\\
&\leq C\theta_0 r_2+C\theta_0 M(0,\rho)+C\theta_0(\varepsilon_2+\sqrt{r_2})r_2^{1-\frac{2m}{q}}.
\end{aligned}
\end{equation}
Putting \eqref{Dbeta}, \eqref{Dalpha1} and \eqref{Dalpha2} together, we obtain
\begin{equation}\label{Du}
 \begin{aligned}
&(\theta_0\tilde\rho)^{1-m}\int_{B^+_{\theta_0\tilde\rho}}|Du|\\
&\leq (\theta_0\tilde\rho)^{1-m}\int_{B_{\theta_0\tilde\rho}}|D\alpha_l|+|D\beta|+|D\phi|\\
&\leq C\theta_0 r_2+C\theta_0 M(0,\rho)+C\theta_0(\varepsilon_2+\sqrt{r_2})r_2^{1-\frac{2m}{q}}\\
&\quad+C\theta_0^{1-m}\Bigg(r_2+(\varepsilon_2+\sqrt{r_2})M(0,\rho)+r_2^{1-\frac{2m}{q}}(\varepsilon_2+\sqrt{r_2})\Bigg)\\
&\quad+C\theta_0^{1-m}(\varepsilon_2+\sqrt{r_2})M(0,\rho)\\
&\leq C(\theta_0^{1-m}(\varepsilon_2+\sqrt{r_2})+\theta_0)M(0,\rho)+C\theta_0^{1-m}r_2\\
&\quad+C\theta_0^{1-m}(\varepsilon_2+\sqrt{r_2})r_2^{1-\frac{2m}{q}}.
\end{aligned}
\end{equation}

Now, we can choose $\theta_0$, $\varepsilon_2$ and $\delta_2$ sufficient small such that
\begin{equation}
 \begin{aligned}
C\theta_0&\leq \frac{1}{6^{m+1}},\\
C\theta_0^{1-m}( 2^{\frac{m}{2}-1}\varepsilon_2+\sqrt{\delta_2})&\leq\frac{1}{6^{m+1}}.
\end{aligned}
\end{equation}
Then, for $\theta_2=\frac12\theta_0$ and $\rho\in(0,r_2)$, we get
\begin{equation}\label{0}
(\theta_2\rho)^{1-m}\int_{B^+_{\theta_2\rho}}|Du|
\leq \frac{1}{6^m} M(0,\rho)+Cr_2+ \frac{1}{6^{m+1}}r_2^{1-\frac{2m}{q}},
\end{equation}
and \eqref{bdy} follows from this and
\begin{equation}
(r_2/2)^{2-m}\int_{B^+_{r_2/2}(x)}\leq 2^{m-2}\varepsilon_2^2,\ \forall x\in T_{r_2/2}
\end{equation} 
%-------------------------------------------------------------------------------------------------------------------------------------------------------------
\begin{comment}
Now, for any given $\alpha\in(0,1)$, we can choose $\theta_0$, $\varepsilon_2$ and $\delta_2$ sufficient small such that
\begin{equation}
 \begin{aligned}
C\theta_0&\leq \frac{1}{3^m}\left(\frac{1}{12}\theta_0\right)^\alpha,\\
C\theta_0^{1-m}( 2^{\frac{m}{2}-1}\varepsilon_2+\sqrt{\delta_2})&\leq\frac{1}{3^m}\left(\frac{1}{12}\theta_0\right)^\alpha.
\end{aligned}
\end{equation}
Then, for $\theta_2=\frac12\theta_0$ and $\rho\in(0,r_2)$, we get
\begin{equation}\label{0}
(\theta_2\rho)^{1-m}\int_{B^+_{\theta_2\rho}}|Du|
\leq \frac{1}{3^{m-1}}\left(\frac{1}{6}\theta_2\right)^\alpha M(0,\rho)+\frac{1}{3^{m-1}}\left(\frac{1}{6}\theta_2\right)^\alpha,
\end{equation}
and \eqref{bdy} follows from this and
\begin{equation}
(r_2/2)^{2-m}\int_{B^+_{r_2/2}(x)}\leq 2^{m-2}\varepsilon_2^2,\ \forall x\in T_{r_2/2}
\end{equation} 
\end{comment}
%-------------------------------------------------------------------------------------------------------------------------------------------------------------
\end{proof}

\begin{proof}[\bf Proof of Lemma \ref{M}]
 We proceed as follows.

\textit{Case 1.} $x_1\in \pt\Omega$. Assume that $x_1=0$ and $\Omega\cap B_{r_1}(x_1)=B_{r_1}^+$. Let $\theta_2$ be given in Lemma \ref{boundary M} and $B_{\rho}(y)\subset B_{\theta_2r_1/6}$. Then we have either $B_\rho(y)\cap T_1\neq\emptyset$ or $B_\rho(y)\cap T_1=\emptyset$.

If $B_\rho(y)\cap T_1\neq\emptyset$, we define $P(y)=(y',0)$ for $y=(y',y_m)$. Then $B_\rho(y)\subset B_{2\rho}(P(y))$.
Thus, by choosing $\varepsilon_1\leq\varepsilon_2$ and $2\delta_1\leq \delta_2$, Lemma \ref{boundary M} gives us that
\begin{equation}
 \begin{aligned}
\rho^{1-m}\int_{B_{\rho}(y)\cap B^+_{\theta_2r_1/6}}|Du|&\leq 2^{m-1}(2\rho)^{1-m}\int_{B_{2\rho}(P(y))\cap B^+_{\theta_2r_1/6}}|Du|\\
&\leq 2^{m-1}\left(\frac{1}{6^m} M(P(y),\rho_2)+Cr_1+ \frac{1}{6^{m+1}}r_1^{1-\frac{2m}{q}}\right)\\
&\leq \frac14 M(0,r_1)+Cr_1+r_1^{1-\frac{2m}{q}},
\end{aligned}
\end{equation}where $2\rho=\theta_2\rho_2$. 

If  $B_\rho(y)\cap T_1=\emptyset$ and $\rho\leq\frac15|y_m|$. It follows from the proof of Lemma 3.4 in \cite{wang2009regularity} that
\begin{equation}
 \begin{aligned}
\rho^{1-m}\int_{B_{\rho}(y)} |Du|&\leq M(y,\rho)\leq M(y,|y_m|/5)\\
&\leq\frac14 M(y,|y_m|)\leq \frac14M(P(y),2|y_m|).
\end{aligned}
\end{equation}
Since $B_{|y_m|}(y)\subset B^+_{2|y_m|}(P(y))\subset B^+_{\theta_2r_1/3}(P(y))\subset B^+_{\theta_2r_1/2}(0)$, we have
\begin{equation}
\rho^{1-m}\int_{B_{\rho}(y)} |Du|\leq \frac14M(0,r_1).
\end{equation}

If  $B_\rho(y)\cap T_1=\emptyset$ and $\rho>\frac15|y_m|$. Then
\begin{equation}
 \begin{aligned}
\rho^{1-m}\int_{B_{\rho}(y)} |Du|&\leq\rho^{1-m}\int_{B^+_{\rho+|y_m|}(P(y))} |Du|\\
&\leq\left(\frac{\rho+|y_m|}{\rho} \right)^{m-1}(\rho+|y_m|)^{1-m}\int_{B^+_{\rho+|y_m|}(P(y))} |Du|\\
&\leq \frac16M(0,r_1)+Cr_1+r_1^{1-\frac{2m}{q}}.
\end{aligned}
\end{equation}

Thus, \eqref{M monotone} holds for $\theta_1=\frac16\theta_2$.

\textit{Case 2.} $x_1\in \Omega$. Let $d_1={\rm dist}(x_1,\pt\Omega)$. If $d_1\geq r_1$, then $B(x_1,r_1)\subset\Omega$ and Lemma 3.4 in \cite{wang2009regularity} implies that for any $\theta_1\in(0,\frac{1}{12}]$, we have
\begin{equation}
M(x_1,\theta_1r_1)\leq \frac14 M(x_1,r_1).
\end{equation}

If $d_1\leq r_1$, then $|x_1-P(x_1)|={\rm dist}(x_1,\pt\Omega)$. Let $k_0$ be sufficiently large (to be determined later). Then
we consider two cases: (a) $r_1/2^{k_0}\leq |x_1-P(x_1)|\leq r_1$; (b) $ |x_1-P(x_1)|<r_1/2^{k_0}$. In case (a), since $B(x_1,r_1/2^{k_0})\subset \Omega$, Lemma 3.4 in \cite{wang2009regularity} implies that for any $\theta_1\in(0,\frac{1}{12}]$, we have
\begin{equation}
M(x_1,\theta_1r_1/2^{k_0})\leq \frac14 M(x_1,r_1/2^{k_0})\leq \frac14 M(x_1,r_1).
\end{equation}
In case (b), for any $0<\delta<1$, note that 
\begin{equation}
B_{\delta r_1}(x_1)\subset B_{(1/2^{k_0}+\delta)r_1}(P(x_1)),  \ \  B_{(1-1/2^{k_0})r_1}(P(x_1))\subset B_{r_1}(x_1).
\end{equation}
Hence, 
\begin{equation}
 \begin{aligned}
&M(x_1, {\delta r_1})\leq M(P(x_1),{(1/2^{k_0}+\delta)r_1}), \\  
&M(P(x_1),{(1-1/2^{k_0})r_1})\leq M(x_1,{r_1}).
\end{aligned}
\end{equation}

Now, we choose $k_0$, $\delta=\frac18\theta_2$ such that $1/2^{k_0}+\frac18\theta_2\leq \frac16\theta_2(1-1/2^{k_0})$, where $\theta_2$ is given by Lemma \ref{boundary M}. Then
\begin{equation}
 \begin{aligned}
 M(x_1, {\theta_2 r_1/8})
&\leq M(P(x_1),{(1/2^{k_0}+\theta_2/8)r_1})\\
&\leq\frac14 M(P(x_1),{(1-1/2^{k_0})r_1})+C_1r_1+r_1^{1-\frac{2m}{q}}\\
&\leq\frac14 M(x_1,{r_1})+C_1r_1+r_1^{1-\frac{2m}{q}}
\end{aligned}
\end{equation}
and Lemma \ref{M} is proved with $\theta_1= \theta_2/2^{k_0}$.

\end{proof}

Now, we can prove Theorem \ref{sebr} to end this section.

\begin{proof}[\bf Proof of Theorem \ref{sebr}]
For any $x\in B_{r_0/2}(x_0)$, $r\in (0,\frac{r_0}{2}]$, \eqref{monotonicity} implies that we can choose $\varepsilon_0$ and $\delta_0$ such that 
\begin{equation}
r^{2-m}\int_{B_r(x)\cap \Omega}|Du|^2\leq \varepsilon_1^2,
\end{equation}
where $\varepsilon_1$ is given in Lemma \ref{M}. Applying Lemma \ref{M}, we know there exists $\theta_1\in (0,1/4)$ such that
\begin{equation}\label{m}
M(x,\theta_1r)\leq \frac14 M(x,r)+C_1r+r^{1-\frac{2m}{q}}.
\end{equation}
for any $x\in B_{r_0/2}(x_0)$, $r\in (0,\frac{r_0}{4}]$. 

For any $r\in (0,\frac{r_0}{4}]$, there exists an integer $k$ such that $\theta_1^{k+1}R<r\leq\theta^{k}R$, where $R=\frac{r_0}{4}$.
It follows from \eqref{m} that
\begin{equation}
 \begin{aligned}
&M(x,r)\\
&\leq M(x,\theta_1^{k}R)\\
&\leq \frac{1}{4^k}M(x,R)+\frac{1}{4^{k-1}}(C_1R+R^{1-\frac{2m}{q}})(1+4\theta_1+\cdots+(4\theta_1)^{k-1})\\
&\leq \exp\left(k\log \frac14\right)M(x,R)+\frac{4}{1-4\theta_1}(C_1R+R^{1-\frac{2m}{q}})\exp\left(k\log\frac14\right)\\
&\leq\left(M(x,R)+\frac{4}{1-4\theta_1}(C_1R+R^{1-\frac{2m}{q}})\right)\exp\left(\left(\frac{\log\frac{r}{R}}{\log\theta_1}-1\right)\log\frac14\right)\\
&\leq\left(4M(x,R)+\frac{16}{1-4\theta_1}(C_1R+R^{1-\frac{2m}{q}})\right)\left(\frac{r}{R}\right)^{\alpha_0}\\
&\leq Cr^{\alpha_0},
\end{aligned}
\end{equation}
where $\alpha_0=\frac{\log\frac14}{\log\theta_1}$. Hence, Theorem \ref{sebr} follows from the  Morrey's decay lemma in 
\cite{giaquinta1983multiple}.

\end{proof}

\vspace{2em}

\section{Proof of Theorem \ref{holder}} %and Theorem \ref{smooth}}

\begin{proof}[Proof of Theorem \ref{holder}]

Let 
\begin{equation}
\mathcal{S}=\{x\in M_{\rho_0}:\liminf_{r\to0}r^{2-m}\int_{B_r(x)\cap  M_{\rho_0}}|Du|^2>\varepsilon_0^2 \}
\end{equation}
where $\varepsilon_0$ is the constant in  Theorem \ref{sebr}. It follows from a standard covering argument that  $H^{m-2}(\mathcal{S})=0$.
For any $x_0\in  M_{\rho_0}\setminus\mathcal{S}$, there exists $r_0\in (0,\delta_0)$, where $\delta_0$ is the constant in  Theorem \ref{sebr}, such that
\begin{equation}
    r_0^{2-m}\int_{M\cap B_{r_0}(x_0)}|\nabla u|^2\leq \varepsilon_0^2.
\end{equation}
Then Theorem \ref{sebr} implies  $u\in C^{\alpha_0}( \bar{M}\cap B_{r_0/2}(x_0),N)$ for some $\alpha_0\in(0,1)$. Hence,  $u\in C^{\alpha_0}( M_{\rho_0}\setminus\mathcal{S},N)$.

For $\phi\in C^{1,1}(\partial M,N)$, we can show that for any $x_0\in \pt M\setminus\mathcal{S}$, it holds
\begin{equation}\label{u22}
    r^{2-m}\int_{M\cap B_{r}(x_0)}|\nabla u|^2\leq C(\beta)r^{2\beta}.
\end{equation}
 where $0<r\leq\tilde r_0< \frac{r_0}{2}$, $0<\beta<1-\frac{2m}{q}$.

 To do so, we define $v\in H^1(B_r(x_0),\mathbb{R}^K)$ as
 \begin{equation}
\left\{
 \begin{aligned}
		&\Delta v=0, \ \text{in} \ B_r(x_0),\\
		&v=u_\phi, \ \text{on} \ \pt B_r(x_0),
\end{aligned}
\right.
 \end{equation}
and the maximum principle implies that
 \begin{equation}
 \|v-u_\phi\|_{L^\infty(B_r(x_0))}\leq {\rm osc}_{B_r(x_0)}u_\phi\leq Cr^{\alpha_0}.
 \end{equation}
Then we have
\begin{equation}\label{uphi-v}
 \begin{aligned}
&\int_{B_r(x_0)}|\nabla (u_\phi-v)|^2\\
&=-\int_{B_r(x_0)}\langle\Delta(u_\phi-v),u_\phi-v\rangle\\
&=-\int_{B_r(x_0)}\langle\Delta u_\phi,u_\phi-v\rangle\\
&\leq\int_{B_r(x_0)}|\Delta u_\phi||u_\phi-v|\\
&\leq Cr^{\alpha_0}\left(\int_{B^+_r(x_0)}|\Delta u|+\int_{B^+_r(x_0)}|\Delta \phi|\right)\\
&\leq Cr^{\alpha_0}\left(\int_{B^+_r(x_0)}|\nabla u|^2+\int_{B^+_r(x_0)}|\nabla u||\psi|^2+\int_{B^+_r(x_0)}|\Delta \phi|\right)\\
&\leq Cr^{\alpha_0}\left(\int_{B^+_r(x_0)}|\nabla u|^2+r^{m-\frac{4m}{q}}+r^{m}\right).
\end{aligned}
\end{equation}
Hence, the standard estimate on harmonic functions yields
\begin{equation*}
 \begin{aligned}
&(\theta r)^{2-m}\int_{B_{\theta r}(x_0)}|\nabla u_\phi|^2\\
&\leq 2\left((\theta r)^{2-m}\int_{B_{\theta r}(x_0)}|\nabla v|^2+(\theta r)^{2-m}\int_{B_{\theta r}(x_0)}|\nabla(u_\phi-v)|^2\right)\\
&\leq 2\left(\theta^2 r^{2-m}\int_{B_{ r}(x_0)}|\nabla v|^2+(\theta r)^{2-m}\int_{B_{\theta r}(x_0)}|\nabla(u_\phi-v)|^2\right)\\
&\leq 4\theta^2 r^{2-m}\int_{B_{ r}(x_0)}|\nabla u_\phi|^2+6(\theta r)^{2-m}\int_{B_{r}(x_0)}|\nabla(u_\phi-v)|^2\\
&\leq 4\theta^2 r^{2-m}\int_{B_{ r}(x_0)}|\nabla u_\phi|^2+C\theta^{2-m}r^{2-m+\alpha_0}\left(\int_{B^+_r(x_0)}|\nabla u|^2+r^{m-\frac{4m}{q}}+r^{m}\right)\\
&\leq 4\theta^2 r^{2-m}\int_{B_{ r}(x_0)}|\nabla u_\phi|^2+C\theta^{2-m}r^{2-m+\alpha_0}\left(\int_{B_r(x_0)}|\nabla u_\phi|^2+r^m+r^{m-\frac{4m}{q}}\right)\\
&\leq (4\theta^2+C\theta^{2-m}r^{\alpha_0}) r^{2-m}\int_{B_{ r}(x_0)}|\nabla u_\phi|^2+C\theta^{2-m}r^{\alpha_0}(r^2+r^{2-\frac{4m}{q}}).
\end{aligned}
\end{equation*}
Now, we can choose $\tilde r_0<\frac{r_0}{2}$ and $\tilde\theta_0=\tilde\theta_0(\tilde r_0,\beta,m)\in (0,\frac12)$ such that for any $0<r\leq \tilde r_0$ we have
\begin{equation}\label{u phi}
(\tilde\theta_0r)^{2-m}\int_{B_{\theta r}(x_0)}|\nabla u_\phi|^2\leq \tilde\theta_0^{2\beta} r^{2-m}\int_{B_{ r}(x_0)}|\nabla u_\phi|^2+\tilde\theta_0^{2\beta}r^{2-\frac{4m}{q}}.
\end{equation}
Then  the iteration of \eqref{u phi} yields
\begin{equation}
 r^{2-m}\int_{B_{ r}(x_0)}|\nabla u_\phi|^2\leq C(\beta)r^{2\beta}
\end{equation}
for $0<\beta<1-\frac{2m}{q}$, and \eqref{u22} follows.

For $0<\beta<1-\frac{2m}{q}$, it follows from \eqref{u22} and \eqref{uphi-v} that
\begin{equation}\label{uphi-v'}
 \begin{aligned}
\int_{B_r(x_0)}|\nabla (u_\phi-v)|^2
&\leq Cr^{\alpha_0}\left(\int_{B^+_r(x_0)}|\nabla u|^2+r^{m-\frac{4m}{q}}\right)\\
&\leq C(\beta)r^{\alpha_0}(r^{m-2+2\beta}+r^{m-\frac{4m}{q}})\\
&\leq C(\beta)r^{m-2+\alpha_0+2\beta}.
\end{aligned}
\end{equation}
Hence, using the Campanato estimate for harmonic functions, we have
\begin{equation}\label{uphi-average}
 \begin{aligned}
&(\theta r)^{-m}\int_{B_{\theta r}(x_0)}|\nabla u_\phi-\overline{(\nabla u_\phi)}_{x_0,\theta r}|^2\\
&\leq 2\left((\theta r)^{-m}\int_{B_{\theta r}(x_0)}|\nabla v-\overline{(\nabla v)}_{x_0,\theta r}|^2+(\theta r)^{-m}\int_{B_{\theta r}(x_0)}|\nabla(u_\phi-v)|^2\right)\\
&\leq 2\left(\theta^2 r^{-m}\int_{B_{r}(x_0)}|\nabla v-\overline{(\nabla v)}_{x_0,r}|^2+(\theta r)^{-m}\int_{B_{\theta r}(x_0)}|\nabla(u_\phi-v)|^2\right)\\
&\leq 4\theta^2 r^{-m}\int_{B_{r}(x_0)}|\nabla u_\phi-\overline{(\nabla u_\phi)}_{x_0,r}|^2+6(\theta r)^{-m}\int_{B_{r}(x_0)}|\nabla(u_\phi-v)|^2\\
&\leq 4\theta^2 r^{-m}\int_{B_{r}(x_0)}|\nabla u_\phi-\overline{(\nabla u_\phi)}_{x_0,r}|^2+ C(\beta)\theta^{-m}r^{\alpha_0+2\beta-2}\\
&\leq \frac12 r^{-m}\int_{B_{r}(x_0)}|\nabla u_\phi-\overline{(\nabla u_\phi)}_{x_0,r}|^2+ C(\beta)r^{2\mu}
\end{aligned}
\end{equation}
provided that we choose $\theta=\frac{1}{2\sqrt{2}}$ and $\mu$ satisfying $r^{\alpha_0+2\beta-2}\leq r^{2\mu}$ for $0<r\leq \tilde r_0$, where we have used the assumption $q>\frac{4m}{\alpha_0}$  and
\begin{equation}
\overline{(\nabla u_\phi)}_{x_0,\theta r}=\frac{1}{|B_r(x_0)|}\int_{B_{\theta r}(x_0)}\nabla u_\phi
\end{equation}
is the average of $\nabla u_\phi$ over $B_{\theta r}(x_0)$. Again, the iteration of \eqref{uphi-average} gives us
\begin{equation}
r^{-m}\int_{B_r(x_0)}|\nabla u_\phi-\overline{(\nabla u_\phi)}_{x_0,r}|^2\leq C r^{2\mu}.
\end{equation}for $0<r\leq \tilde r_0$.
This, combined with the interior regularity in \cite{wang2009regularity}, implies $u\in C^{1,\mu}( M_{\rho_0}\setminus\mathcal{S},N)$.

\end{proof}

%\begin{proof}[Proof of Theorem  \ref{smooth}]

%\end{proof}
\vspace{1em}

% ----------------------------------------------------------------
\vspace{2em}
\bibliographystyle{amsplain}
\bibliography{reference}

\end{document}